\documentclass[12pt]{amsart}
\usepackage{amsthm,amssymb,amsmath}
\usepackage{hyperref}
\usepackage[capitalize]{cleveref}
\usepackage[margin=1in]{geometry}
\usepackage{enumitem}
\def\hP{{\hat P}}
\def\h0{{\hat 0}}
\def\Gr{{\rm Gr}}
\def\M{{\mathcal{M}}}
\def\N{{\mathcal{N}}}
\def\R{{\mathbb{R}}}
\def\env{{\rm env}}
\def\rank{{\rm rank}}
\def\sp{{\rm span}}
\def\oPi{{\mathring{\Pi}}}
\def\ker{{\rm ker}}
\def\det{{\rm det}}
\def\GL{{\rm GL}}

\def\codim{{\rm codim}}
\def\I{{\mathcal{I}}}
\newtheorem{theorem}{Theorem}
\newtheorem{proposition}[theorem]{Proposition}
\newtheorem{lemma}[theorem]{Lemma}
\newtheorem{corollary}[theorem]{Corollary}
\newtheorem{conjecture}[theorem]{Conjecture}
\newtheorem{definition}[theorem]{Definition}
\newtheorem{example}[theorem]{Example}
\newtheorem{remark}[theorem]{Remark}

\begin{document}
\title{On the face stratification of the $m=2$ amplituhedron}
\author{Thomas Lam}
\address{Department of Mathematics, University of Michigan, 2074 East Hall, 530 Church Street, Ann Arbor, MI 48109-1043, USA}
\email{\href{mailto:tfylam@umich.edu}{tfylam@umich.edu}}
\begin{abstract}
We define and study the face stratification of the $m=2$ amplituhedron.  We show that the face poset is an upper order ideal in the face poset of the totally nonnegative Grassmannian.  Our construction is consistent with earlier work of Lukowski, and we confirm various predictions of Lukowski.
\end{abstract}
\maketitle

\section{Introduction}
Let $\Gr(k,n)_{\geq 0}$ denote the totally nonnegative Grassmannian \cite{Pos,Lus}, the subspace of the Grassmannian $\Gr(k,n)$ of $k$-planes in $\R^n$ with nonnegative Pl\"ucker coordinates.  The Grassmannian has a positroid stratification \cite{KLS}
$
\Gr(k,n) = \bigsqcup_\M \oPi_\M
$
into open positroid varieties indexed by rank $k$ positroids $\M$.  Intersecting this stratification with $\Gr(k,n)_{\geq 0}$ we obtain the face stratification $\Gr(k,n)_{\geq 0} = \bigsqcup_\M \Pi_{\M,>0}$ into positroid cells \cite{Pos}.  This endows $\Gr(k,n)_{\geq 0}$ with the structure of a regular CW-complex homeomorphic to a closed ball \cite{GKL1,GKL3}.  

For $m \leq n-k$, the amplituhedron $A_{n,k,m}$ is the image of $\Gr(k,n)_{\geq 0}$ under a linear projection $Z: \Gr(k,n) \to \Gr(k,k+m)$ where $Z$ is represented by a $n \times (k+m)$ matrix with positive $(k+m) \times (k+m)$ minors.  When $k=1$, $A_{n,1,m}$ is the familiar cyclic polytope.  The amplituhedron was defined by Arkani-Hamed and Trnka \cite{AT} to study and construct super Yang-Mills amplitudes.  While the case $m=4$ is of the most physical significance, the $m=2$ case is an important toy model which has been the subject of much recent work \cite{BH,LPW,Luk,PSBW,LPSV,RST}.  The $m=2$ amplituhedron is the focus of this paper. 

We define (\cref{def:face}) the face stratification of $A_{n,k,2}$ as the intersection of $A_{n,k,2}$ with the positroid stratification of $\Gr(2,n)$ under the \emph{twistor embedding} of $A_{n,k,2}$ into $\Gr(2,n)$.  This definition is directly analogous to Postnikov's definition of the face stratification of $\Gr(k,n)_{\geq 0}$.  We further show in \cref{thm:semi} that our face stratification agrees with the stratification of $A_{n,k,2}$ as a semialgebraic set, or as a (conjectural) positive geometry.  

In \cref{thm:main}, we give a complete description of the face poset $P_{n,k}$ of $A_{n,k,2}$ as an upper order ideal in the face poset $Q_{n,2}$ of $\Gr(2,n)_{\geq 0}$.  As a key combinatorial tool we define a \emph{twistor map} on the level of matroids (\cref{sec:twistor}).  

Lukowski \cite{Luk} studied the boundaries of $A_{n,k,2}$ in a recursive and computational way.  As we explain in \cref{sec:Luk}, our approach is consistent with his.  In particular, we confirm conjectures from \cite{Luk}, for the rank generating function of $P_{n,k}$ (\cref{thm:corank}) and that the face poset is Eulerian (\cref{thm:Eulerian}).  Our results are also consistent with the recent work of Ranestad, Sinn, and Telen \cite{RST} who studied the case $k=m=2$; see \cref{rem:RST}.

Denote $[n]:=\{1,2,\ldots,n\}$ and let $\binom{[n]}{k}$ denote the $k$-element subsets of $[n]$.  Let $\leq_a$ be the total order on $[n]$ that is the cyclic rotation of the usual order, but with minimum $a$.  For $k$-element subsets $I = \{i_1 <_a i_2 <_a \cdots <_a i_k\}$ and $J = \{j_1 <_a j_2 <_a \cdots <_a j_k\}$ we have $I \leq_a J$ if and only if $i_s \leq_a j_s$ for all $s$.

\section{The amplituhedron in twistor space}
For a point $C \in \Gr(k,n)$ or a $k\times n$ matrix, we let $\Delta_I(C)$ denote Pl\"ucker coordinate labeled by the subset $I \in \binom{[n]}{k}$.  If $I =  \{i_1 < \cdots < i_a\}, J = \{j_1 < \cdots < j_b\}$ are two subsets of $[n]$ with $a+b=k$, we let $\Delta_{IJ}(C)$ be the signed Pl\"ucker coordinate: it is 0 if $I \cap J \neq \emptyset$; otherwise it is equal to $\pm \Delta_{I \cup J}(C)$, the determinant of the $k \times k$ matrix whose columns are given by the columns $c_{i_1},\ldots,c_{i_a},c_{j_1},\ldots,c_{j_b}$ of $C$, taken in this order. 

Let $1 \leq k \leq n$ and $m \leq n-k$ be even.  Let $Z$ be a $n \times (k+m)$ \emph{positive} matrix, that is, the $(k+m) \times (k+m)$ minors of $Z$ are positive.  We denote by $\sp(Z) \in \Gr(k+m,n)$ the column span of $Z$.
We define the \emph{amplituhedron (in twistor space)} $A_{n,k,m} = A_{n,k,m}(Z)$ as the image of $\Gr(k,n)_{\geq 0}$ under the twistor map
$$
\varphi_m: \Gr(k,n) \to \Gr(m,n)
$$
given by
$$
\Delta_I(\varphi_m(C)) =  \sum_{J \in \binom{[n]}{k}} \Delta_J(C) \Delta_{J I}(Z) =: \langle I \rangle.
$$
The Pl\"ucker coordinate $\langle I \rangle$ is called a twistor coordinate of $C$.  The amplituhedron depends on the choice of $Z$, but for simplicity the dependence on $Z$ is often suppressed from our notation.  Note that $\varphi_m$ is only a rational map.  It is not defined on $V \in \Gr(k,n)$ exactly when $\dim(V \cap \ker(Z)) \geq 1$.  The amplituhedron in twistor space was called the ``B-amplituhedron" in  \cite{KWZ} and shown to be isomorphic to the amplituhedron of \cite{AT} there.

The twistor map $\varphi_m$ is the composition of the linear projection map 
$$
Z: \Gr(k,n) \to \Gr(k,k+m), \qquad V \mapsto Z(V)
$$
with the map $\psi: \Gr(k,k+m) \to \Gr(m,n)$ given by
$$
\Delta_I(\psi(Y)) = \det(YZ_{i_1}Z_{i_2}\cdots Z_{i_m})  \qquad \text{ for } Y \in \Gr(k,k+m),
$$
where for an $m$-element subset $I = \{i_1,\ldots,i_m\}$, the notation $\det(YZ_{i_1}Z_{i_2}\cdots Z_{i_m})$ denotes the $(k+m) \times (k+m)$ determinant obtained by concatenating the $k \times (k+m)$ matrix $Y$ with the $m$ row vectors $Z_{i_1},\ldots,Z_{i_m} \in \R^{k+m}$.  See for example \cite[Section 18]{LamCDM} and \cite[Section 3.1]{PSBW}.

\begin{lemma}\label{lem:twistor}
For $Z$ full rank, the map $\psi: \Gr(k,k+m) \to \Gr(m,n)$ is an embedding.
\end{lemma}
\begin{proof}
The group $\GL(n)$ acts simultaneously on $Z$ and on $\Gr(m,n)$.  Using it, we may assume that $Z$ is the identity matrix in the first $(k+m)$ rows.  That is, $Z_i = e_i \in \R^n$ for $i=1,2,\ldots,k+m$ and $Z_i = 0$ for $i = k+m-1,\ldots,n$.  In this case, for $I \in \binom{[k+m]}{m}$, the twistor coordinate $\Delta_I(\psi(Y))$ is, up to a predictable sign, the usual Pl\"ucker coordinate $\Delta_{[k+m]\setminus I}(Y)$.
\end{proof}

It follows from Lemma \ref{lem:twistor} that the amplituhedron in twistor space is isomorphic to the usual amplituhedron, defined as the image $Z(\Gr(k,n)_{\geq 0})$ inside $\Gr(k,k+m)$.  

Arkani-Hamed, Thomas, and Trnka \cite{ATT} gave a conjectural description of $A_{n,k,m}$ using inequalities of twistor coordinates, and a topological condition on sign-flips of twistor coordinates.  For $m=2$, this description, together with a comprehensive description of triangulations, is available due to the works of Bao--He, Karp--Williams, Lukowski--Parisi--Williams, and Parisi--Sherman-Bennet--Williams \cite{BH,KW,LPW,PSBW}.

The complement of the $n$ divisors $\{\langle a (a+1) \rangle\ = 0\}$ in $\Gr(2,n)$ is the top open positroid variety $\oPi(2,n)$ which is a cluster variety of type $A_{n-3}$.  The following result is a special case of upcoming work of Galashin and Lam studying connected components of real cluster varieties.  Let $\oPi_{+}$ be the union of those connected components of $\oPi(2,n)$ that satisfy $\langle 12 \rangle >0, \langle 23 \rangle >0, \ldots, \langle (n-1)n \rangle >0$.

\begin{proposition}[\cite{GLconncomp}]\label{prop:GL}
The space $\oPi_+$ has $n-1$ connected components.  The $k$-th connected component $\oPi_{+,k}$ is the closure of the locus in $\oPi_+$ where the sequence $\langle 12 \rangle, \langle 13 \rangle, \ldots, \langle 1n \rangle$ has $k$ sign-flips (and no zeroes).
\end{proposition}

In \cref{prop:GL}, the sign of the remaining cyclic minor $\langle 1n \rangle$ on a connected component $\oPi_{+,k}$ is determined by the parity of $k$, the number of sign-flips.  

\begin{remark}
The space $\oPi_{+,k}$ and the space $\mathcal{G}^\circ_{n,k,2}$ of \cite[Section 11.2]{PSBW} have similar definitions, and the two spaces have the same closure in $\Gr(2,n)$.  
The spaces $\oPi_{+,k}$ also appear in the work of Dian and Heslop \cite{DH}; see also \cite{EHM}.
\end{remark}

The following result follows from \cite[Proposition 11.11]{PSBW}; see also \cite{BH,ATT}.

\begin{theorem}\label{thm:sign-flip}
We have 
$$A_{n,k,2}(Z)=\overline{\oPi_{+,k} \cap \Gr(2,\sp(Z))}.$$
\end{theorem}

\begin{remark}
If we consider the family of amplituhedra $A_{n,k,m}(Z)$ as $Z$ varies over $\Gr(k+m,n)_{\geq 0}$, we obtain the universal amplituhedron \cite{GLparity}.  We believe that it is still an open problem to determine whether $\bigcup_{\text{positive } Z} A_{n,k,2}(Z)$ is dense in $\oPi_{+,k}$.  
\end{remark}

We note that $\Gr(2,n)_{> 0}$ is the connected component $\oPi_{+,0}$.  Motivated by \cref{prop:GL} and \cref{thm:sign-flip}, we study the intersection of $A_{n,k,2}$ with the positroid stratification of $\Gr(2,n)$ in \cref{def:face} below.  We begin with the matroid combinatorics of the twistor map.

\section{The twistor map on matroids}\label{sec:twistor}
Let $\M$ be a matroid of rank $k$ on $[n]$.  We will think of a matroid as a collection of bases.  For a point $C \in \Gr(k,n)$, the matroid $\M_C$ of $C$ is given by
$$
\M_C :=\left \{ I \in \binom{[n]}{k} \mid \Delta_I(C) \neq 0 \right\} \subseteq \binom{[n]}{k}.
$$
We let $\M^* := \{[n] \setminus I\mid I \in \M\}$ denote the dual matroid of rank $n-k$ on $[n]$.  A matroid $\M$ is a \emph{positroid} if it is the matroid $\M= \M_C$ of a totally nonnegative point $C \in \Gr(k,n)_{\geq 0}$ of the Grassmannian.  Positroids were completely classified in the works of Postnikov \cite{Pos} and Oh \cite{Oh}.  We describe the classification in terms of Grassmann necklaces.

A \emph{(k,n)-Grassmann necklace} $\I = (I_1,\ldots,I_n)$ is a collection of $k$-subsets $I_a \in \binom{[n]}{k}$ satisfying the conditions: (a) if $a \in I_a$ then $I_{a+1} = I_a \setminus \{a\} \cup \{a'\}$ for some $a'$, and (b) if $a \notin I_a$ then $I_{a+1} = I_a$.  (Here, the index $a$ is to be taken modulo $n$.)

Let $\M$ be a rank $k$ matroid on $[n]$, and let $I_a = I_a(\M)$ be the (lexicographically) minimal base of $\M$ with respect to $\leq_a$.  Then $\I(\M) = (I_1,\ldots,I_n)$ is a $(k,n)$-Grassmann necklace.  The map $\M \mapsto \I(\M)$ restricts to a bijection between positroids of rank $k$ on $[n]$ and $(k,n)$-Grassmann necklaces, and we let $\I \to \M_\I$ denote the inverse of this map.  We define the positroid envelope $\env(\M)$ of a matroid $\M$ to be the (unique) smallest positroid containing $\M$.  Alternatively, the envelope $\env(\M) = \M_{\I(\M)}$ is the unique positroid with the same Grassmann necklace as $\M$.  Thus positroids are the maximal elements among matroids with the same Grassmann necklace.

We deonte by $Q_{n,k}$ the poset of positroids of rank $k$ on $[n]$, ordered by inclusion.  The poset $Q_{n,k}$ is isomorphic to the dual of a lower order ideal in the Bruhat order of the affine symmetric group.  We refer the reader to \cite{Pos,Oh,KLS} for more on positroids.

\begin{definition}\label{def:twistormat} Let $m \leq n-k$.  Define the $m$-twistor matroid $\M^{\downarrow m}$ by
$$
\M^{\downarrow m} = \left\{I \in \binom{[n]}{m} \mid I \subset J \text{ for some } J \in \M^*\right\}.
$$
Alternatively, $\M^{\downarrow m}$ is the $m$-truncation of the dual matroid $\M^*$.
\end{definition}

\begin{proposition}
The twistor $\M^{\downarrow m}$ is a matroid.
\end{proposition}
\begin{proof}
The bases of $\M^{\downarrow m}$ are the independent sets of $\M^*$ of size $m$.  The exchange axiom for bases of $\M^{\downarrow m}$ follows immediately from the independent set axioms for $\M^*$.
\end{proof}

The following result is immediate from the definition.

\begin{proposition}\label{prop:twistGN}
Let $\M$ be a matroid.  Let the Grassmann necklace of $\M^*$ be $(J_1,\ldots,J_n)$.  Then the positroid $\env(\M^{\downarrow m})$ has Grassmann necklace $(I_1,\ldots,I_n)$ where $I_a$ consists of the $m$ minimal elements of $J_a$ with respect to $\leq_a$.
\end{proposition}

\begin{example}Let $k =2$ and $n =5$.  Consider the matroid $\M= \{13,15,34,35,45\}$.  That is, $2$ is a loop and $1,4$ are parallel.  Then $\M^*=\{245,234,125,124,123\}$ and $\M^{\downarrow 2}=\binom{[5]}{2} - \{35\}$.  The Grassmann necklaces of $\M, \M^*$, and $\M^{\downarrow 2}$ are $(13,34,34,45,51)$, $(123,234,342,452,512)$ and $(12,23,34,45,51)$ respectively.
\end{example}

 Since envelope, duality and taking independent sets of size $m$ all preserve inclusions of matroids, we have the following result.

\begin{corollary}\label{cor:orderpres}
The map $\M \mapsto \env(\M^{\downarrow m}) $ is an order-preserving map from $Q_{n,k}$ to $Q_{n,m}$.
\end{corollary}

Positroid envelope commutes with duality.

\begin{proposition}\label{prop:envcomm}
Let $\M$ be a matroid.  Then $\env(\M)^* = \env(\M^*)$.
\end{proposition}
\begin{proof}
The \emph{dual Grassmann necklace} of $\M$ is the sequence $(I_1,\ldots,I_n)$, where $I_a \in \M$ is the basis that is maximal with respect to $\leq_a$.  Both the Grassmann necklace and dual Grassmann necklace uniquely determine a positroid.  The positroid envelope $\env(\M)$ of $\M$ is the unique positroid with dual Grassmann necklace equal to $(I_1,\ldots,I_n)$.  The dual matroid $\M^*$ has (usual) Grassmann necklace $(J_1,\ldots,J_n)$ where $J_a = [n] \setminus I_a$.  So both $\env(\M)^*$ and $\env(\M^*)$ are positroids with Grassmann necklace equal to $(J_1,\ldots,J_n)$.
\end{proof}

\begin{proposition}\label{prop:envenv}
We have $\env(\M^{\downarrow m}) = \env((\env(\M))^{\downarrow m})$.
\end{proposition}
\begin{proof}
By \cref{prop:envcomm}, $\M^*$ and $\env(\M)^*$ have the same Grassmann necklace.  By \cref{prop:twistGN}, it follows that $\M^{\downarrow m}$ and $(\env(\M))^{\downarrow m}$ have the same Grassmann necklace, and thus the same positroid envelope.
\end{proof}

\begin{definition}\label{def:Pnkm}
Let $P_{n,k,m} \subset Q_{n,m}$ to be the subposet of positroids $\env(\M^{\downarrow m})$, ordered by inclusion, as $\M$ varies over $Q_{n,k}$.\end{definition}
By \cref{prop:envenv}, in \cref{def:Pnkm} we could equivalently have let $\M$ vary over all matroids of rank $k$ on $[n]$.

We will be particularly interested in the case $m = 2$.  In this case, we write $P_{n,k}$ for $P_{n,k,2}$.  This is a subposet of $Q_{n,2}$.  Rank $2$ positroids $\N$ on $[n]$ have the following direct description.  A subset $L \subset [n]$ are loops.  A number of disjoint cyclic intervals $[a_1,b_1],\ldots,[a_r,b_r]$ are specified such that $\rank([a_i,b_i]) = 1$.  In other words, all elements in $[a_i,b_i]$ that are not loops are parallel.  We use the notation $\N= (L,\{[a_1,b_1],\ldots,[a_r,b_r]\})$ for elements of $Q_{n,2}$.  We always assume in this notation that $r$ is taken to be minimal and each $[a_i,b_i]$ is taken as small as possible.  For example, if $|[a_i,b_i]\setminus L| = 1$ then this cyclic interval can be omitted; if $a_i \in L$ or $b_i \in L$ then the cyclic interval $[a_i,b_i]$ can be made smaller.

Let $\N = (L,\{[a_1,b_1],\ldots,[a_r,b_r]\}) \in Q_{n,2}$ and $\N' = (L',\{[a'_1,b'_1],\ldots,[a'_s,b'_s]\}) \in Q_{n,2}$.  Then $\N \leq \N'$ if and only if $L' \subseteq L$ and for each cyclic interval $[a'_i,b'_i]$, either there exists a cyclic interval $[a_j,b_j]$ such that $[a'_i,b'_i] \subseteq [a_j,b_j]$, or $|[a'_i,b'_i] \setminus L| \leq 1$.

For $\N= (L,\{[a_1,b_1],\ldots,[a_r,b_r]\}) \in Q_{n,2}$, let $S = S(\N):= \bigcup_i [a_i,b_i]$, and define
\begin{equation}\label{eq:cd}
d(\N):= 2k +r - |S \setminus L| - 2|L|, \;\; c(\N) := 2k-d(\N), \;\;  \text{and} \;\; e(\N) := r+k - |S \cup L|.
\end{equation}
(Note that $e(\N)$ depends on $k$.)  

\begin{definition}\label{def:N2}
Suppose that $\N \in Q_{n,2}$ satisfies $e(\N) \geq 0$.  We define a matroid $(\N^{\uparrow k})^*$ of rank $n-k$ on the groundset $[n]$ as follows.  Begin with the uniform matroid of rank $n-k$ on the groundset $([n] \setminus (S \cup L)) \cup \{a_1,\ldots,a_r\}$, which has cardinality $e(\N) +n- k \geq n-k$.  For each $i=1,2,\ldots,r$, add elements labeled by $(a_i,b_i] \setminus L$ parallel to $a_i$.  Then add loops labeled by the set $L \subset [n]$.  
\end{definition}

The rank $k$ matroid $\N^{\uparrow k}$ is defined to be the dual of $(\N^{\uparrow k})^*$.  It can be described as the largest matroid satisfying the following rank conditions.  The elements in $L$ are coloops of $\N^{\uparrow k}$.  For each cyclic interval $[a_i,b_i]$, let $A = [a_i,b_i] \setminus L$.  Then $$\rank_{\N^{\uparrow k}}([n]\setminus A)= k - |A|+1.$$
It is straightforward to check that $(\N^{\uparrow k})^{\downarrow 2} = \N$.

\begin{lemma}
Suppose that $\N \in Q_{n,2}$ satisfies $e(\N) \geq 0$.  Then $\N^{\uparrow k}$ and $(\N^{\uparrow k})^*$ are positroids.
\end{lemma}
\begin{proof}
The uniform matroid is a positroid.  Adding cyclically consecutive parallel elements, or adding loops, to a matroid preserves the property of being a positroid.  It follows that $(\N^{\uparrow k})^*$ is a positroid.  Since the class of positroids is closed under duality \cite{OhFPSAC}, $\N^{\uparrow k}$ is also a positroid.
\end{proof}

Recall that a poset is graded if all maximal chains have the same length.

\begin{theorem}\label{thm:Pnk}
Suppose that $n-k \geq 2$.  Then 
\begin{enumerate}
\item $P_{n,k}$ is an upper order ideal of $Q_{n,2}$ with $\binom{n}{k}$ minimal elements $(L,\{\})$ where $|L| = k$,
\item $P_{n,k} = \{\N \in Q_{n,2} \mid e(\N) \geq 0\}$,
\item $P_{n,k}$ is graded of height $2k$ and has rank function $d(\N)$ and corank function $c(\N)$.
\end{enumerate}
\end{theorem}

\begin{proof}
We show (2). 
Suppose that $\env(\M^{\downarrow 2}) = \N$.  By \cref{prop:twistGN}, the loops of $\M^*$ are exactly the loops of $\N$ since $a$ is a loop of $\N$ if and only if $a \notin I_a(\N)$.  Let $\M^*$ have Grassmann necklace $(J_1,\ldots,J_n)$.  We must have $\rank_{\M^*}([a_i,b_i]) = 1$ for otherwise one of the $J_a$ (with $a \in [a_i,b_i]$) would contain two elements in $[a_i,b_i]$.  All these rank conditions are disjoint.  Let $\M^*$ be any matroid that satisfies these rank conditions and has loops equal to $L$.  There are $n - |S \cup L|$ elements of the ground set for which there are no conditions imposed.  Any independent set of $\M^*$ consists of some of these elements together with at most one element in each cyclic interval $[a_i,b_i]$.  It follows that the rank of $\M^*$ is bounded above by $r + n- |S \cup L|$, and this is greater than or equal to $n-k$ if and only if $e(\N) \geq 0$.    Conversely, the construction of \cref{def:N2} shows that when $e(\N) \geq 0$, a matroid $\M^*$ of rank $n-k$ exists satisfying these conditions.  This proves (2).   

It is easy to see, from the description of the partial order on $Q_{n,2}$ in terms of the encoding $\N= (L,\{[a_1,b_1],\ldots,[a_r,b_r]\})$, that the condition $e(\N) \geq 0$ defines an upper order ideal in $Q_{n,2}$.  If $a \in S \setminus L$ then changing $a$ to a loop does not change the value of $e(\N)$.  It follows from this that the minimal elements of $P_{n,k}$ are of the form $(L,\{\})$, and $e(L,\{\}) = 0$ exactly when $|L|=k$.  This proves (1).  

Statement (3) also follows since $Q_{n,2}$ itself is a graded poset with rank function equal to ${\rm rk}(\N) = 2(n-2) +r - |S \setminus L| - 2|L|$.
\end{proof}

We remark that the minimal elements described in \cref{thm:Pnk}(1) can be identified (by removing loops) with the rank 2 uniform matroids on $[n] \setminus L$.

\begin{remark}
Let $\N= (L, \{[a_1,b_1],\ldots,[a_r,b_r]\}) \in P_{n,k}$.  The positroid $(\N^{\uparrow k})^*$ of \cref{def:N2} can also be described as follows.  
Let $I_a \in \binom{[n]}{n-k}$ be the $\leq_a$-minimal subset satisfying the two conditions: $I_a \cap L = \emptyset$ and $|I_a \cap [a_i,b_i]| \leq 1$ for $i=1,2,\ldots,r$.  Then one can check directly that $\I = (I_1,I_2,\ldots,I_n)$ is a Grassmann necklace, and we have $(\N^{\uparrow k})^* = \M_{\I}$. 

Let $P \subset \R^n$ be the polytope in the subspace $H_{n-k} = \{(x_1,\ldots,x_n) \mid x_1+ \cdots + x_n = n-k\} \subset \R^n$ cut out by the equalities and inequalities
\begin{align*}
&0 \leq x_i \leq 1 & \mbox{for $i \in [n]$},\\
&x_\ell = 0 &\mbox{for $\ell \in L$},\\
&0 \leq x_{a_i} + x_{a_i+1} + \cdots + x_{b_i} \leq 1 &\mbox{for $i=1,2,\ldots,r$.}
\end{align*}
Then $P$ is the matroid polytope of $(\N^{\uparrow k})^*$ and $P$ is an alcoved polytope, generalized permutohedron, and polypositroid \cite{LP2}.
 \end{remark}

\section{The face stratification of the $m=2$ amplituhedron}

For a positroid $\M \in Q_{n,k}$, let $\Pi_\M$, $\oPi_\M$, and $\Pi_{\M,>0}$ denote the positroid variety, open positroid variety, and open (totally nonnegative) positroid cell respectively \cite{Pos,KLS,LamCDM}.  \cref{def:twistormat} of the twistor map on matroids is motivated by the following result.

\begin{lemma}\label{lem:down}
Suppose that $C \in \Pi_\M$ and $\varphi$ is defined on $C$.  Then $\varphi_m(C) \in \Pi_{\env(\M^{\downarrow m})}$.
\end{lemma}
\begin{proof}
We have
$$
\langle I \rangle =\sum_{K} \Delta_K(C) \Delta_{K I}(Z) = \sum_{K \in \M} \Delta_K(C) \Delta_{K I}(Z)
$$
where for the second equality we have used that $\Delta_K(C)=0$ for $K \notin \M$ when $C \in \Pi_\M$.  We have $I \notin \M^{\downarrow m}$ if and only if $I \cap K \neq \emptyset$ for all $K \in \M$.  In this case, $\Delta_{KI}(Z) = 0$ for all $K \in \M$, and thus $\langle I \rangle  = 0$ for $I \notin \M^{\downarrow m}$.   In particular, $\langle I \rangle = 0$ for $I \notin \env(\M^{\downarrow m})$.  It follows that $\varphi_m(C) \in \Pi_{\M^{\downarrow m}}$.
\end{proof}

\begin{proposition}\label{prop:boundarysign}
Let $\M$ be a rank $k$ positroid on $[n]$, and let $\N = \M^{\downarrow m}$.  Suppose that $\N$ has Grassmann necklace $(I_1,\ldots,I_n)$.  Then for each 
$a \in [n]$, the twistor coordinate $\langle I_a \rangle$ takes constant sign (and never vanishes) on the face $\Pi_{\M,>0}$.
\end{proposition}
\begin{proof}
Let $C \in \Pi_{\M,>0}$.  We have
$$
\langle I_a \rangle = \sum_{K \in \M} \Delta_K(C) \Delta_{K I_a}(Z).
$$
Let the Grassmann necklace of $\M^*$ be $(J_1,\ldots,J_n)$.  By \cref{prop:twistGN}, we have that $I_a$ are the $m$ smallest elements of $J_a$ with respect to $\leq_a$.  Now let $K_a := [n] \setminus J_a$, which is the maximal base of $\M$ with respect to $\leq_a$.  Then $K_a$ must contain all the holes in $I_a$: that is all elements $b \in [n]\setminus I_a$ such that $b <_a i$ for some $i \in I_a$.  Indeed, by maximality of $K_a$, this is true for all $K \in \M$ such that $K \cap I_a = \emptyset$.  It follows that for such $K$, the sign of $\Delta_{K I_a}(Z)$ depends only on $\N$ and $a$.  We know that $\Delta_K(C) \geq 0$ for all $K$ and $\Delta_{K_a}(C) >0$.  Thus $\langle I_a \rangle \neq 0$, and the sign of $\langle I_a \rangle$ depends only on $\N$ and $a$. 
\end{proof}

Let $\M = \M_\I$ have Grassmann necklace $(I_1,\ldots,I_n)$.  Then $\oPi_\M$ is the open subset of $\Pi_\M$ where $\Delta_{I_a}$ does not vanish for any $a \in [n]$.

\begin{corollary}\label{cor:varphi}
Let  $\M\in Q_{n,k}$ and $\N= \M^{\downarrow m}$.  Then
$ \varphi_m( \Pi_{\M,>0}) \subset \oPi_{\env(\N)}$.
\end{corollary}

Now, assume that $m=2$.  Motivated by the preceding results, we define the face strata of $A_{n,k,2}$ as follows.
\begin{definition}\label{def:face}
The face stratification of $A_{n,k,2}$ is the intersection of $A_{n,k,2}$ with the positroid stratification of $\Gr(2,n)$.
\end{definition}

\begin{remark}
For even $m>2$, the strata defined in \cref{def:face} could still be interesting (cf. \cite[Section 9]{GLparity}).  However, we would not recover the full combinatorics of faces of $A_{n,k,m}$.  For $m=4$, there are facets defined by the hyperplanes $\langle i i+1 j j+1 \rangle = 0$, which are not part of the positroid stratification.  It would be interesting to study the stratification on $\Gr(4,n)$ generated by all such hyperplanes.  One could speculate that the face poset of $A_{n,k,4}$ is a subposet of this stratification (and similarly for other $m$). 
\end{remark}

For $\N \in Q_{n,2}$, let $A_\N:= A_{n,k,2} \cap \oPi_{\N}$ be the intersection of the amplituhedron in twistor space with the open positroid variety.  Note that we are considering real, but not necessarily positive, points in $\oPi_{\N}$.

\begin{theorem}\label{thm:main}
We have $A_{\N} \neq \emptyset$ if and only if $\N \in P_{n,k}$.  Furthermore, we have 
\begin{equation}
\label{eq:closure}
\overline{A_{\N}} = \bigsqcup_{\N' \leq \N} A_{\N'}.
\end{equation}
\end{theorem}

\begin{proof}
By the positroid stratification $
\Gr(k,n)_{\geq 0} = \bigsqcup_\M \Pi_{\M,>0}
$ and 
\cref{cor:varphi}, we have 
$$
A_{n,k,2}= \varphi_2(\Gr(k,n)_{\geq 0}) = \bigcup_{\M \in Q_{n,k}} \varphi_2(\Pi_{\M,>0}) \subset \bigcup_{\N \in P_{n,k}} \oPi_\N.
$$
Let $\N \in P_{n,k}$ and $\M = \N^{\uparrow k}$.  Then $\varphi_2(\Pi_{\M,>0}) \subset A_{\N}$, establishing the first statement.  We have
\begin{equation}\label{eq:closure2}
\overline{A_{\N}} = \overline{\varphi_2(\bigcup_{\tilde \M \mid \env(\tilde \M^{\downarrow 2})= \N} \Pi_{\tilde \M,>0})} = \varphi_2(\overline{\Pi_{\M,>0}}) = \varphi_2(\bigcup_{\M' \leq \M} \Pi_{\M,>0}) =  \bigcup_{\N' \leq \N} A_{\N'}.
\end{equation}
For the second equality, we have used $\overline{\Pi_{\M,>0}} = \bigsqcup_{\M' \leq \M} \Pi_{\M,>0}$, where the union is over positroids $\M'$ contained in $\M$.
For the last equality, specifically the containment $\supset$, we have used that $\N' \leq \N$ implies $(\N')^{\uparrow k} \leq \N^{\uparrow k}$, and that $\M' = \N'^{\uparrow k}$ is the largest matroid satisfying $\env((\M')^{\downarrow 2}) = \N'$.
\end{proof}

By \cref{thm:semi}(2) below, for generic $Z$, the intersection $\oPi_{\N} \cap \Gr(2,\sp(Z))$ has dimension $d(\N)$ and codimension $c(\N)$ inside $\Gr(2,\sp(Z))$.  We expect that (regardless of whether $Z$ is generic) $A_\N$ is full-dimensional in $\oPi_{\N} \cap \Gr(2,\sp(Z))$.

\begin{conjecture}\label{conj:main}
The stratum $A_\N$ (resp. $\overline{A_\N}$) is an open (resp. closed) ball of dimension $d(\N)$.  The space $A_{n,k,2}$, together with the stratification $\{A_\N\}$ is a regular CW-complex homeomorphic to a closed ball of dimension $2k$, and furthermore, it is a positive geometry \cite{LamOPAC,ABL}.
\end{conjecture}

In the case $k=2$, the ``positive geometry" part of \cref{conj:main} follows from \cite{RST}.

Suppose that $X_{\geq 0}$ is a closed semialgebraic subset of an irreducible real algebraic variety $X$ such that $X$ is the Zariski closure of $X_{\geq 0}$.
We define $X_{>0}$ to be the interior of $X_{\geq 0}$ (in the analytic topology on $X$) and the boundary $\partial X_{\geq 0}$ to be $X_{\geq 0} \setminus X_{>0}$.  The \emph{algebraic boundary} $\partial_a X_{\geq 0}$ is the Zariski-closure of $\partial X_{\geq 0}$.  The facets of $X_{\geq 0}$ are the semialgebraic sets obtained by intersecting $\partial X_{\geq 0}$ with the irreducible components of $\partial_a X_{\geq 0}$.  The semialgebraic faces of $X_{\geq 0}$ are the collection of semialgebraic sets obtained by recursively taking facets.

\begin{theorem}\label{thm:semi}
Suppose that $Z$ is generic.  Then 
\begin{enumerate}[label=(\alph*)]
\item for $\N \in P_{n,k}$, we have $\dim(A_\N) = \dim  \overline{A_\N} = d(\N) = \dim(\Pi_\N \cap \Gr(2,\sp(Z)))$.
\item $A_\N$ is the interior of $\overline{A_\N}$, 
\item $\partial \overline{A_\N} = \bigcup_{\N' < \N} A_{\N'}$,
\item $A_\N$ is a connected component of $\oPi_\N \cap \Gr(2,\sp(Z))$.
\end{enumerate}
The semialgebraic faces of $A_{n,k,2}$ agrees with the stratification defined in \cref{def:face}.
\end{theorem}

\begin{proof}
The dimension estimate (a) will be proved in \cref{sec:dim}.  We prove (b), (c), (d) assuming (a).
We first consider $A_{n,k,2}$ itself, which has Zariski-closure the Grassmannian $\Gr(2,\sp(Z))$.  We show that the Zariski-closure $\partial_a A_{n,k,2}$ of $\partial A_{n,k,2}$ in $\Gr(2,\sp(Z))$ is the union of the $n$ positroid divisors $\langle 12 \rangle = 0, \langle 23 \rangle = 0, \ldots, \langle n1 \rangle = 0$ intersected with $\Gr(2,\sp(Z))$.  We follow the argument \cite[Proof of Proposition 3.1]{RST}.  First, the ``triangulations" of $A_{n,k,2}$ from \cite{BH,PSBW} show that $A_{n,k,2}$ is the closure of an open subset of $\Gr(2,\sp(Z))$.  Thus $A_{n,k,2}$ is the closure of its interior.  It follows as in \cite[Proof of Proposition 3.1]{RST} that the boundary of $A_{n,k,2}$ is pure of codimension 1.  The only possible divisors appearing in $\partial_a A_{n,k,2}$ are the vanishing loci of $\langle 12 \rangle, \langle 23 \rangle, \ldots, \langle n1 \rangle$ or $\langle 13 \rangle, \langle 14 \rangle, \ldots$, since these are the inequalities that appear in \cref{thm:sign-flip}.  However, by cyclic symmetry \cref{thm:sign-flip} can also be formulated using sign-flips of the sequence $\langle 23 \rangle, \langle 24 \rangle,\ldots$; see also \cite[Section 5]{ATT}.  Since we know the boundary is codimension one, only $\langle 12 \rangle, \langle 23 \rangle, \ldots, \langle n1 \rangle$ can appear.  Indeed, all of them appear -- the divisor $\langle i (i+1) \rangle$ contains $A_\N$ where $\N = (\{\},\{[i,i+1]\})$ satisfies $c(\N) = 1$.

Given a list $X_1,\ldots,X_r$ of subvarieties of a variety $X$, we can consider the stratification generated by $\{X_1,\ldots,X_r\}$: we repeatedly intersect the subvarieties, breaking the intersections into irreducible components, and repeat.  For the $n$ positroid divisors of $\Gr(2,n)$, one obtains the positroid stratification \cite{KLS}.  

By \cite[Theorem 5.13]{KLS}, the positroid variety $\Pi_\M$ is normal.  By Kleiman transversality \cite[Remark 7]{Kle}, for a general $Z$ (i.e., for $Z$ a matrix belonging to a dense Zariski-open subset of all matrices), the intersection $\Gr(2,\sp(Z)) \cap \Pi_\M$ is normal, and by the argument in \cite[Remark 3.4.11(1)]{FOV} it is irreducible (except possibly when the intersection has expected dimension 0, in which case it can be checked directly).  Thus the stratification generated by $\partial_a A_{n,k,2}$ is the stratification in $\Gr(2,\sp(Z))$ induced by intersecting with the positroid stratification of $\Gr(2,n)$.  For simplicity, we still call this the positroid stratification.

The semialgebraic faces of $A_{n,k,2}$ are generated by recursively intersecting boundaries of faces of $A_{n,k,2}$ with the various pieces of the positroid stratification.  (Of course, many such intersections will be empty.) By (a), $A_\N$ is Zariski dense in $\Pi_\N \cap \Gr(2,Z)$.  By \eqref{eq:closure2} and the connectedness of $\Pi_{\M,>0}$, the face $A_\N$ is connected.  Since the only boundaries of $A_{n,k,2}$ are the positroid divisors, $A_\N$ must thus be a connected component of $\oPi_\N \cap \Gr(2,\sp(Z))$, proving (d).  It also follows that the strata $\overline{A_\N}$ are the only candidates for the faces of $A_{n,k,2}$.

Now, since $A_\N$ is a connected component, it is open in $\Pi_\N \cap \Gr(2,\sp(Z))$, and it must be contained in the interior of $\overline{A_\N}$.  On the other hand, it follows from \cref{prop:boundarysign} that $\bigcup_{\N' < \N} A_{\N'}$ must belong to the boundary of $A_\N$.  Statements (b) and (c) follow.  We conclude by induction on $c(\N)$ that each $\overline{A_\N}$ is part of the semialgebraic face stratification of $A_{n,k,2}$, proving the last statement.
\end{proof}

Let $(I_1,\ldots,I_n)$ be the Grassmann necklace of $\N$.  Let $F_a:= \overline{A_\N} \cap H_a$ denote the intersection of $\overline{A_\N}$ with the hyperplane $H_a = \{\langle I_a \rangle = 0\}$.
Note that $F_a$ is always contained in the boundary of $\overline{A_\N}$, but may not be codimension one.  For example, if $\N = (\{1,2\},\{[5,6]\})$ and $k=3$, then the cocovers of $\N$ are $(\{1,2,5\},\{\})$ and $(\{1,2,6\},\{\})$.  Since $I_3 = \{3,4\}$, neither of these strata belong to $F_3$.  Also, it is possible for $F_a$ to have multiple components, as is already the case for $\Gr(2,n)_{\geq 0}$.  For example, if $\N= (\emptyset,\{[1,2]\})$ then $F_1$ contains both of the strata indexed by $(\{1\},\{\})$ and $(\emptyset,\{[1,3]\})$.  However, we do have the following.

\begin{proposition}
For $\N' \lessdot \N$, we have $\overline{A_{\N'}} \subseteq F_a$ for some $a$.
\end{proposition}
\begin{proof}
One of the subsets $I'_a$ in the Grassmann necklace of $\N'$ differs from the corresponding subset $I_a$ in the Grassmann necklace of $\N$.  In fact, we have $I'_a >_a I_a$.
\end{proof}

It follows from \eqref{eq:closure} that the union of the facets of $\overline{A_\N}$ (that is, $\overline{A_{\N'}}$ for $\N' \lessdot \N$) is equal to the union $\bigcup_a F_a$.

\begin{remark} In general, the set of faces of a closed semialgebraic space could have many pathologies.  However, \cref{thm:semi} says that for $A_{n,k,2}$ the closed faces form a stratification in the following sense: the intersection of two closed faces is a union of closed faces.
\end{remark}

\begin{remark}
\cref{def:face} and \cref{thm:semi} says that the face stratification of $A_{n,k,2}$ is analogous to the following recursive definition of the faces of a convex polytope $P$.  Define the facets of $P$ as the \emph{codimension one} intersections $P \cap H$ with a hyperplane $H$, such that $P$ lies inside a closed halfspace $H_{\geq 0}$ on one side of $H$.  Recursively taking facets of facets and so on, we obtain the list of all faces of $P$. 
\end{remark}

\section{The poset $P_{n,k}$}

For $\N= (L,\{[a_1,b_1],\ldots,[a_r,b_r]\}) \in Q_{n,2}$, define the set $T :=\bigcup_i (\{a_i,a_{i}+1,\ldots,b_i-1\} \setminus L)$.  Then $\N$ is uniquely determined by the pair $(L,T)$.  We shall use this notation below.  
\subsection{Lukowski matrices}\label{sec:Luk}

Our definition of face stratification of $A_{n,k,2}$ coincides with the ``boundaries" predicted in earlier work of Lukowski \cite{Luk}.  For each $\N \in P_{n,k}$ define the following $k \times n$ \emph{Lukowski matrix}.  For each loop $a \in L$, place a row with non-zero entry only in the $a$-th position.  For each $a \in T$, we add a row with non-zero entries in the $(a,a')$ position, where $a'$ is the (cyclically) smallest index greater than $a$ which is not a loop.  Any remaining rows are filled with generic entries.

\begin{example}
Let $k=5$, $n=8$, $L=\{3\}$, and $T=\{2,4,6\}$.  Then Lukowski's matrix is
$$
\begin{bmatrix}
0 & * &0 &* &0&0&0&0 \\
0 & 0 &* &0 &0&0&0&0 \\
0 & 0 &0 &* &*&0&0&0 \\
0 & 0 &0 &0 &0&*&*&0 \\
* & * &* &* &*&*&*&* 
\end{bmatrix}.
$$
\end{example}

The following result shows that our construction is consistent with \cite{Luk}.
\begin{proposition}
A generic Lukowski matrix has matroid equal to $\N^{\uparrow k}$.
\end{proposition}
\begin{proof}
Let $X$ be such a matrix.  For generic entries, it is clear that $\rank([n]- a) = k-1$ for $a \in L$.  Thus each $a \in L$ is a coloop in the matrix.  Contracting all such coloops (that is, removing the $a$-th column and corresponding row), we may assume that $L = \emptyset$.

If $\rank([a,b]) = 1$ is one of the defining rank conditions of $\N$, then we see that the columns of $X$ in positions $[n] \setminus [a,b]$ are all concentrated in $k+1-|[a,b]|$ rows.  Thus the matroid of $X$ satisfies the rank condition $\rank([n]-[a,b]) \leq k+1-|[a,b]|$.  It is easy to see that for generic entries equality holds, and furthermore, the matroid of $X$ is maximal given these rank conditions, and thus is equal to $\N^{\uparrow k}$.
\end{proof}

\subsection{Rank-generating function}
The following result confirms the conjecture of \cite[Section 2.8]{Luk}.
\begin{theorem}\label{thm:corank}
The corank generating function of $P_{n,k}$ is given by
$$
\sum_{\N \in P_{n,k}} t^{c(\N)} = \sum_{c=0}^k \binom{n}{c} t^c (1+t)^c.
$$
\end{theorem}
\begin{proof}
First, assume that $\N = (\emptyset, T)$.  Then $e(\N) = r+k - |S|$, and since $|T| = |S|-r$, the condition that $e(\N) \geq 0$ is equivalent to $|T| \leq k$.  Any subset $T \subset [n]$ satisfying $0 \leq |T| \leq k$ gives an element of $P_{n,k}$.  This gives the corank generating function $\sum_{T} t^{|T|} = \sum_{a=0}^k \binom{n}{a} t^{a}$.  For general $L$, we may first fix $L$, and then choose the subset $T$ in $[n] \setminus L$.   Summing over $b = |L|$, we get
\begin{equation*}
\sum_{b=0}^k \binom{n}{b} t^{2b} \sum_{a=0}^{k-b} \binom{n-b}{a} t^{a} = \sum_{c=0}^k \binom{n}{c} t^c \sum_{b=0}^c \binom{c}{b} t^b =  \sum_{c=0}^k \binom{n}{c} t^c (1+t)^c. \qedhere
\end{equation*}
\end{proof}

\begin{remark}\label{rem:RST}
Ranestad, Sinn, and Telen \cite{RST} study many aspects of the $k=2, m=2$ amplituhedron, including the face stratification.  Our poset $P_{n,2}$ should be compared to the strata labeled ``b" in  \cite[Table 1]{RST}.  In particular, taking $k = 2$ in \cref{thm:corank}, we get 
$$
1 + n t(1+ t) + \binom{n}{2} t^2 (1+t)^2 = 1+ n t + (n+\binom{n}{2}) t^2 + n(n-1) t^3 + \binom{n}{2} t^4,
$$
which agrees with the numerology in \cite[Table 1]{RST}.  The ``r" strata in \cite[Table 1]{RST} correspond to elements in $Q_{n,2}$ not belonging to $P_{n,2}$.
\end{remark}

\subsection{Eulerian-ness}

Let $\hP_{n,k}$ be the poset obtained from $P_{n,k}$ by adjoining a minimum element $\h0$.  Then $\hP_{n,k}$ is a graded poset with a maximum and a minimum.  A graded poset is \emph{Eulerian} if intervals of rank greater than one have an equal number of even rank and odd rank elements.  The following result confirms a prediction of Lukowski \cite{Luk}.

\def\J{{\mathcal{J}}}
\begin{theorem}\label{thm:Eulerian}
The poset $\hP_{n,k}$ is Eulerian.
\end{theorem}
\begin{proof}
We need to check that intervals of rank greater than one have an equal number of even rank and odd rank elements.
For intervals $[x,y] \subset \hP_{n,k}$ where $x \neq \h0$, this is known since such intervals are intervals in $Q_{n,k}$, and, $Q_{n,k}$ is Eulerian; see \cite{KLS,Wil}.  We consider intervals $[\h0,\N]$ in $\hP_{n,k}$ where $\N=(L, \{[a_1,b_1],\ldots,[a_r,b_r]\}) = (L,T)$.  If $|L|>0$, then removing the elements that are loops, we find that $[\h0,\N]$ is isomorphic to an interval of the same form in $\hP_{n-|L|,k-|L|}$.  Henceforth we assume that $\N = (\emptyset,\{[a_1,b_1],\ldots,[a_r,b_r]\})= (\emptyset, T)$ has no loops. 

Let $\J(\N)$ be the principal lower order ideal in $P_{n,k}$ generated by $\N$.  Let $\N' = (L',T') \in P_{n,k}$.
We have $\N'\in \J(\N)$ if and only if for each $i$, the complement $[a_i,b_i] \setminus (L'\cup T')]$ is either a single element $c$, and we have $A_i:=[c+1,b_i] \subset L'$, or it is empty and we set $A_i:=\emptyset$.
We say that an element $d \in (L' \cup T')$ is \emph{flippable} if moving $d$ from $L'$ to $T'$ or vice versa produces another $\N'' \in \J(\N)$.  Every element in $T'$ is flippable.  An element $d \in L'$ is not flippable exactly when $d$ belongs to one of the sets $A_i$.  Call $\N'=(L',T')$ \emph{rigid} if it contains no flippable elements.  Note that the set of flippable elements does not change after performing a flip.  We may thus define an involution on the set of non-rigid positroids in $\J(\N)$ by flipping the smallest (in the usual order on $[n]$) flippable element.

Since flipping changes the parity of the rank of an element, it remains to consider the set of rigid elements.  This consists of a single element $(L',T')$ where $T'$ is empty and $L' = \bigcup_i [a_i+1,b_i]$, which has even rank in $P_{n,k}$.  It follows that the interval $[\h0,\N]$ in $\hP_{n,k}$ has an equal number of odd rank and even rank elements.
\end{proof}

\def\hQ{{\hat Q}}
\begin{remark}
Williams \cite{Wil} proves that $\hQ_{n,k}$ is thin and shellable and thus the face poset of a regular CW-complex homeomorphic to a closed ball.  Since $P_{n,k}$ is thin, it is easy to see that $\hP_{n,k}$ is thin.  We expect $\hP_{n,k}$ to be EL-shellable.
\end{remark}

\section{Dimension estimates}\label{sec:dim}
In this section, we suppose that $Z$ is generic.  Note that a generic positive $Z$ is a generic $Z$ in the sense of \cite{LamAS}.  In \cite{KLS}, the cohomology class $[\Pi_\M] \in H^*(\Gr(k,n))$ of a positroid variety is described in terms of \emph{affine Stanley symmetric functions} \cite{LamASold}.  We give an explicit description of this function in the special case that $\M = \N^{\uparrow k}$.  We use notation from \cite{KLS,LamAS,LamCDM}.  

\begin{proposition}\label{prop:LamAS}
Let $\M= \N^{\uparrow k}$, where $\N=(L,\{[a_1,b_1],\ldots,[a_r,b_r]\})$, and set $$d_i:=|[a_i,b_i]\setminus L|-1, \qquad \text{for } i=1,2,\ldots,r.$$  Then the cohomology class $[\Pi_\M]$ is the image in $H^*(\Gr(k,n))$ of the symmetric function
$$
h_{n-k}^{|L|} \prod_{i=1}^r s_{(n-k-1)^{d_i}}
$$ 
where $h_r$ and $s_\lambda$ denote homogeneous and Schur symmetric functions, and $(n-k-1)^{d_i}$ denotes a $d_i \times (n-k-1)$ rectangle.
\end{proposition}
\begin{proof}
This result is a variant of \cite[Proposition 10.15]{LamCDM}.  We calculate the class of $[\Pi_{\M^*}] \in H^*(\Gr(n-k,n))$.  The class of $[\Pi_\M]$ is obtained by taking the transpose $s_\lambda \mapsto s_{\lambda'}$ of the Schur function expansion of $[\Pi_{\M^*}]$.  Setting $\ell \in L$ to be a loop in $\M^*$ cuts out a subGrassmannian of the Grassmannian $\Gr(n-k,n)$ with cohomology class the elementary symmetric function $e_{n-k}$.  Replacing $n$ with $n - |L|$, we may assume that $\M^*$ has no loops.  The rank condition $\rank(a,a+1,\ldots,b) \leq 1$ cuts out a Schubert variety with cohomology class $s_{(b-a-1)^{n-k-1}}$.  As we vary over the cyclic intervals $[a_i,b_i]$, these Schubert varieties are transverse (because the intervals are disjoint).  We deduce that the $[\Pi_{\M^*}]$ is represented by the symmetric function $e_{n-k}^{|L|} \prod_{i=1}^r s_{(d_i)^{n-k-1}}$.  Taking the transpose gives the stated result for $[\Pi_\M]$.
\end{proof}
It is also possible to prove \cref{prop:LamAS} using the combinatorics of affine Stanley symmetric functions.

In \cite{LamAS}, the cohomology class of the twistor image $\varphi_m(\Pi_\M) \subset \Gr(m,n)$ (or equivalently, the image $Z(\Pi_\M) \subset \Gr(k,k+m)$) was computed, again using affine Stanley symmetric functions.  The result there was only stated in the case that $\Pi_\M$ and $\varphi_m(\Pi_\M)$ have the same dimension.   We state a slight modification of \cite[Proposition 3.5]{LamAS} which holds when $\Pi_\M$ and $\varphi_m(\Pi_\M)$ have different dimensions.

  For $\lambda=(\lambda_1,\ldots,\lambda_k)$ a partition, let $r_m(\lambda) = \sum_{i=1}^k \max(\lambda_i - (n-k-m),0)$.
\begin{proposition}\label{prop:dimX}
Let $f = \sum_{\lambda \subseteq (n-k)^k} c_\lambda s_\lambda$ represent the class of an irreducible subvariety $X \subset \Gr(k,n)$ and set $C(X) = \{\lambda \mid c_\lambda > 0\}$.  Then for a generic $Z$, the image $Z(X) \subset \Gr(k,k+m)$ has codimension 
$$
\codim(Z(X)) = \min\{r_m(\lambda) \mid \lambda \in C(X)\}.
$$
\end{proposition}
\begin{proof} We sketch the proof, assuming the reader is familiar with the geometry of the map $Z: \Gr(k,n) \to \Gr(k,k+m)$ as discussed in \cite{LamAS}.  To compute the codimension of $Z(X)$, we check whether it intersects a generic translate of various Schubert varieties $Y_\mu \subset \Gr(k,k+m)$, where $[Y_\mu] = s_\mu$.  The (closure of the) inverse image $Z^{-1}(Y_\mu)$ is a Schubert variety in $\Gr(k,n)$ with the same cohomology class $s_\mu$.  Here, $\mu$ is a partition in a $k \times m$ box.  Now, generic translates of two Schubert varieties $X_\lambda$ and $X_\mu$ in $\Gr(k,n)$ intersect if and only if the 180 degree rotation of $\mu$ fits inside the complement of $\lambda$ inside the rectangle $(n-k)^k$.  Thus, $Z(X)$ has codimension $\leq r$ if and only if there is a $\lambda \in C(X)$ such that the intersection of $\lambda$ with the rightmost $m$ columns of the $(n-k) \times k$ rectangle contains $\leq r$ boxes.  This intersection has $r_m(\lambda)$ boxes, from which the result follows.
\end{proof}

\begin{proposition}\label{prop:dimest}
With the same notation as in \cref{prop:LamAS},  the codimension of $\varphi_m(\Pi_M)$ in $\Gr(2,\sp(Z))$ is equal to $c(\N)$ defined in \eqref{eq:cd}.
\end{proposition}
\begin{proof}
This result follows from \cref{prop:LamAS} and \cref{prop:dimX}.  Indeed, the Schur function expansion of $h_{n-k}^{|L|} \prod_{i=1}^r s_{(n-k-1)^{d_i}}$ contains the Schur function $s_\lambda$ where $\lambda= \lambda(\N)$ contains $|L|$ parts equal to $n-k$ and $\sum_{i=1}^r d_i$ parts equal to $n-k-1$. For this partition, we have $r_2(\lambda(\N)) = 2|L| + \sum_{i=1}^r d_i = c(\N)$.  It follows from the Littlewood-Richardson rule that all Schur functions $s_\mu$ appearing in the expansion of $h_{n-k}^{|L|} \prod_{i=1}^r s_{(n-k-1)^{d_i}}$ have partitions $\mu$ with less than or equal to $|L| + \sum_{i=1}^{r} d_i$ parts.  Thus $\lambda(\N)$ minimizes $r_m(\mu)$ among all $s_\mu$ appearing in the expansion of $h_{n-k}^{|L|} \prod_{i=1}^r s_{(n-k-1)^{d_i}}$.
\end{proof}

\begin{proof}[Proof of \cref{thm:semi}(1)]
By \cref{prop:LamAS}, or its proof, or \cite[Proposition 10.15]{LamCDM}, the cohomology class $[\Pi_\N] \in H^*(\Gr(2,n))$ is given by the symmetric function
\begin{equation}\label{eq:PiN}
[\Pi_\N]= e_{2}^{|L|} \prod_{i=1}^r h_{d_i}
\end{equation}
With $Z$ generic, the intersection $\Pi_\N \cap \Gr(2,\sp(Z))$ has cohomology class obtained by taking only those Schur functions that fit inside a $2 \times k$ rectangle from \eqref{eq:PiN}.  As long as this class is non-zero, the intersection would have codimension equal to the degree of $[\Pi_\N]$, which is $2|L| + \sum_{i=1}^r d_i = c(\N)$.  This class is always nonzero when $\N \in P_{n,k}$, but there do exist $\N \in Q_{n,2} \setminus P_{n,k}$ for which the class is nonzero and the intersection $\Pi_\N \cap \Gr(2,\sp(Z))$ is non-empty; see \cref{prop:residual}.

Now let $\M = \N^{\uparrow k}$.  Then by \cref{lem:down}, we have $\varphi_m(\Pi_\M) \subset \Pi_\N \cap \Gr(2,\sp(Z))$ and by \cref{prop:dimest} it has codimension $c(\N)$ and dimension $d(\N)$.  Since $\Pi_{\M,>0}$ is Zariski-dense in $\Pi_\M$ \cite{KLS}, it follows that the dimension of $\varphi_m(\Pi_{\M,>0})$ is equal to $d(\N)$.  By \eqref{eq:closure2}, we have $\overline{A_\N} = \overline{\varphi_m(\Pi_{\M,>0})}$ and the result follows.
\end{proof}

We also describe the \emph{residual arrangement} (cf. \cite{RST}) of $A_{n,k,2}$.  These are the positroid strata $\Pi_\N$ that intersect $\Gr(2,\sp(Z))$ which are \emph{not} faces of $A_{n,k,2}$, or equivalently, do not belong to $P_{n,k}$.

\begin{proposition}\label{prop:residual}
Let $\N \in Q_{n,2} - P_{n,k}$.  With notation as in \cref{prop:LamAS}, $\Pi_\N$ belongs to the residual arrangement if and only if we have (a) $c(\N) \leq 2k$, (b) $a:= k - |L| \geq 0$, and (c) $d_i \leq a$ for $i=1,2,\ldots, r$.
\end{proposition}
\begin{proof}
As in the proof of \cref{thm:semi}(1), we have that $\Pi_\N$ belongs to the the residual arrangement if and only if the symmetric function \eqref{eq:PiN}, when expanded into Schur functions, contains a Schur function that fits inside a $2 \times k$ rectangle.  The conditions (a) and (b) follow immediately, and for (c) we note that $\prod_{i=1}^r h_{d_i}$ contains a Schur function that fits inside a $2 \times a$ rectangle if and only if $d_i \leq a$ for all $i$, and $\sum_{i=1}^r d_i \leq 2a$.
\end{proof}

\medskip

\noindent
{\bf Acknowledgements.} 
We acknowledge support from the National Science Foundation under DMS-2348799 and DMS-1953852 and from the Simons Foundation for a Simons Fellowship.  We thank Gabriele Dian, Paul Heslop, Tomek Lukowski, Matteo Parisi, Kristian Ranestad, and Simon Telen for comments on an earlier version of this manuscript.  We thank the referee for many helpful suggestions on this work.

\end{document}